\documentclass[11pt]{amsart}
                                                                                                                                                                                                                                                                                                                               \usepackage{amssymb,amsthm,amsmath,mathrsfs}
\usepackage{enumerate}
\usepackage{fancyvrb,fancyhdr,geometry}
\makeatletter
\@namedef{subjclassname@2020}{%
  \textup{2020} Mathematics Subject Classification}
\makeatother
\newtheorem{lemma}{Lemma}[section]
\newtheorem{theorem}{Theorem}[section]

\newtheorem{definition}{Definition}[section]
\newtheorem{remark}{Remark}[section]
\newtheorem{example}{Example}[section]
\DeclareMathOperator{\spa}{span}
\begin{document}

\title{generalized norm retrieval and generalized phase retrieval  in Hilbert spaces}

\author[Gh. Rahimlou]{Gholamreza Rahimlou$^{*}$}%
\address{Department of Basic Sciences, Technical and Vocational University (TVU), Tehran, Iran.}
\email{grahimlou@gmail.com}
\email{ghrahimlo@tvu.ac.ir}

\author[V. Sadri]{Vahid Sadri}
\address{Department of Basic Sciences, Technical and Vocational University (TVU), Tehran, Iran.}
\email{vsadri@tvu.ac.ir}

%\date{}
\begin{abstract}
In this paper, we introduce generalized phase retrieval (briefly, g-phase retrieval) and generalized norm retrieval (briefly, g-norm retrieval). Then, we  present some properties of these concepts in Hilbert spaces, with special emphasis on $\Bbb R^{n}$. Finally, the stability of g-norm retrieval for $\Bbb R^{n}$ is proved.
\end{abstract}


\keywords{Frame, g-frames, dual g-frame, norm retrieval, phase retrieval
*Corresponding Author.}

\maketitle

\section{Introduction}
The concept of frames in a separable Hilbert space was originally introduced by Duffin and Schaeffer in the context of non-harmonic Fourier series \cite{ds}. Frames are redundant sets of vectors that provide a stable but non-unique representation for each vector in the space  \cite{ev,ch,naf,Rahim,Rahim1,sun,sad1}. Generalized frames (g-frames) and g-Riesz bases were introduced by Sun \cite{sun} and further developed in subsequent works.

Phase retrieval and norm retrieval are among the most actively studied topics in applied harmonic analysis. Phase retrieval for Hilbert space frames was introduced in \cite{balan} and has since become a significant area of research\cite{Casaza,Wang}. In general, phase retrieval aims to recover a signal (up to a unimodular constant, e.g., a sign in the real case or a complex phase in the complex case) from the magnitudes of its measurements obtained through a redundant linear system. Norm retrieval, on the other hand, focuses on recovering only the norm of the signal from such measurements.

Although much work has been devoted to complex infinite-dimensional phase retrieval, only a few papers have addressed infinite-dimensional real phase retrieval or norm retrieval. In this paper, we introduce two new generalizations: \emph{generalized phase retrieval} (g-phase retrieval) and \emph{generalized norm retrieval} (g-norm retrieval). These notions extend the classical framework by replacing vectors with bounded linear operators between Hilbert spaces, thus providing a unified approach that encompasses both standard frames and g-frames.

We present several properties of g-phase retrieval and g-norm retrieval in general Hilbert spaces, with a particular focus on the finite-dimensional Euclidean space $\mathbb{R}^n$. Among the main contributions, we establish the stability of g-norm retrieval for $\mathbb{R}^n$. More precisely, we show that if a family of operators performs g-norm retrieval up to a small error, then it is possible to construct a nearby family that performs exact g-norm retrieval. This stability result is of practical importance for applications such as signal processing and phase retrieval with noisy measurements.

Throughout this paper, $H$ and $K$ denote separable Hilbert spaces, and $\mathcal{B}(H,K)$ is the set of all bounded linear operators from $H$ to $K$. When $K=H$, we write $\mathcal{B}(H)$ instead of $\mathcal{B}(H,H)$. Let $\mathrm{GL}(H)$ be the set of all $U\in\mathcal{B}(H)$ that have a bounded inverse. For $U\in\mathcal{B}(H,K)$, $\mathcal{R}(U)$ stands for the range of $U$. Given a closed subspace $W$ of $H$, $P_W$ denotes the orthogonal projection onto $W$. Finally, $\{H_i\}_{i\in I}$ is a sequence of Hilbert spaces indexed by a countable set $\mathbb{I}\subset\mathbb{Z}$.

\begin{definition}[\textbf{frame}]
Let $\{f_i\}_{i\in\Bbb I}$ be a sequence of members of $H$. We say that $\{f_i\}_{i\in\Bbb I}$ is a frame  for $H$ if there exist $0<A\leq B<\infty$ such that for each $f\in H$,
\begin{align}\label{ord}
A\Vert f\Vert^2\leq\sum_{i\in\Bbb I}\vert\langle f,f_i\rangle\vert^2\leq B\Vert f\Vert^2.
\end{align}
\end{definition}
The constants $A$ and $B$ are called frame bounds. If the right hand of \eqref{ord} holds, we say that $\{f_i\}_{i\in\Bbb I}$ is a Bessel sequence with bound $B$.
We say that $\{f_i\}_{i\in\Bbb I}$ is a $A$-tight frame (or just tight frame) for $H$, if we have
$$\sum_{i\in\Bbb I}\vert\langle f,f_i\rangle\vert^2=A\Vert f\Vert^2,$$
for each $f\in H$. If $A=1$, the set $\{f_i\}_{i\in\Bbb I}$ is called a Parseval frame.
\begin{definition}[\textbf{g-frame}]
A family $\Lambda:=\lbrace \Lambda_i\in\mathcal{B}(H,H_i)\rbrace_{i\in\Bbb I}$ is called a g-frame for $H$ with respect to $\lbrace H_i\rbrace_{i\in\Bbb I}$,  if there exist $0<A\leq B<\infty$ such that
\begin{equation} \label{1}
A\Vert x\Vert^2\leq\sum_{i\in\Bbb I}\Vert \Lambda_{i}x\Vert^2\leq B\Vert x\Vert^2, \ \  x\in H.
\end{equation}
\end{definition}

Throughout this paper, let $\Lambda = \{\Lambda_i \in \mathcal{B}(H, H_i) : i \in I\}$. 
The constants $A$ and $B$ are called the lower and upper g-frame bounds for $\Lambda$, respectively. 
We say that $\Lambda$ is a g-Bessel sequence with bound $B$ if only the right-hand side inequality of \eqref{1} holds. 
$\Lambda$ is called a tight g-frame if $A = B$.Also, $\Lambda$ is called a Parseval g-frame if $A=B =1$.

If $\Lambda$ is a g-Bessel sequence, then the synthesis and analysis operators are defined by (for more details, we refer to \cite{sun,naf,sad1})
\begin{align*}
T_{\Lambda}: &(\sum_{i\in\Bbb I}\oplus H_i)_{\ell^{2}} \rightarrow H , \qquad T_{\Lambda}^{*}: H \rightarrow (\sum_{i\in\Bbb I}\oplus H_i)_{\ell^{2}},\\
T_{\Lambda}(\lbrace x_i&\rbrace_{i\in\Bbb I})=\sum_{i\in\Bbb I}\Lambda_{i}^{\ast}(x_i) , \qquad T_{\Lambda}^{\ast}(x)=\lbrace \Lambda_{i}x\rbrace_{i\in\Bbb I},
\end{align*}
where, the notation  $(\sum_{i\in\Bbb I}\oplus H_i)_{\ell^{2}}$ will indicate the space
\begin{align*}
\big\lbrace\lbrace x_i\rbrace_{i\in\Bbb I} \ \vert \ x_i\in H_i \ , \ \sum_{i\in\Bbb I}\Vert x_i\Vert^2<\infty \big\rbrace.
\end{align*}
It is a Hilbert space with pointwise operations, and inner product defined by
$$ \langle \{x_i\}, \{y_i\} \rangle :=\sum_{i\in\Bbb I} \langle x_i, y_i \rangle,  \ \ \{x_i\}, \{y_i\} \in (\sum_{i\in\Bbb I}\oplus H_i)_{\ell^{2}}. $$
Now, the g-frame operator is given by
$$S_{\Lambda}x=T_{\Lambda}T^*_{\Lambda}x=\sum_{i\in\Bbb I}\Lambda^{\ast}_{i}\Lambda_{i}x, \ \ x\in H.$$
So, the g-frame operator is positive, self-adjoint and invertible. So, we can get
\begin{align}\label{f1}
A I\leq S_{\Lambda}\leq B I,
\end{align}
and
\begin{align}\label{f2}
x=\sum_{i\in\Bbb I}\Lambda_{i}^{\ast}\Lambda_{i}S_{\Lambda}^{-1}x,\qquad (x\in H).
\end{align}
It is obvious that when $\Lambda$ is a g-Bessel sequence with the bound $B$, then for each $i\in \Bbb I$,
$$\Vert\Lambda_i\Vert=\sup_{\Vert x\Vert=1}\Vert\Lambda_i x\Vert\leq\sqrt{B},$$
therefore, we have
$$\Vert\Lambda_i^*(h_i)\Vert\leq\sqrt{B}\Vert h_i\Vert,\qquad (\forall h_i\in H_i).$$
The following is a result for a stability of g-frames.
\begin{lemma}[\cite{naf}]\label{lemn}
Let $\Lambda$ is a g-frame for $H$ with bounds $A,B$ and $\Theta:=\lbrace \Theta_i\in\mathcal{B}(H,H_i)\rbrace_{i\in\Bbb I}$ be a family of operators. If there exist a $0<R<A$ such that
\begin{align}\label{lemn2}
\sum_{i\in\Bbb I}\Vert\Lambda_i x-\Theta_i x\Vert^2\leq R\Vert x\Vert^2,\qquad (\forall x\in H),
\end{align}
then, $\Theta$ is a g-frame for $H$ with bounds $(\sqrt{A}-\sqrt{R})^2$ and $(\sqrt{B}+\sqrt{R})^2$.
\end{lemma}
Notice that, in Lemma \ref{lemn}, when $\Lambda$ is a just g-Bessel sequence for $H$ with a bound $B$ and there exists a $R>0$ such that the relation \eqref{lemn2} holds, then $\Theta$ is a g-Bessel sequence too with a bound $(\sqrt{B}+\sqrt{R})^2$.
%%%%%%%%%%%%%%%%%%%%%%%
\section{generalized norm and generalized phase retrieval }
This section is devoted to the definitions and basic properties of generalized norm retrieval and generalized phase retrieval. After establishing the necessary notation, we examine how these concepts relate to classical norm retrieval and phase retrieval.
\begin{definition}
We call that $\Lambda$ is  a generalized phase retrieval (or briefly g-phase retrieval)  if for each $x,y\in H$ satisfying $\Vert\Lambda_i x\Vert=\Vert\Lambda_i y\Vert$ for all $i\in\Bbb I$, then there exist some scalar $\lambda$ such that $\vert\lambda\vert=1$ and $x=\lambda y$.
\end{definition}
\begin{definition}
We say that $\Lambda$ is  a generalized norm retrieval (or briefly g-norm retrieval)  if for each $x,y\in H$ satisfying $\Vert\Lambda_i x\Vert=\Vert\Lambda_i y\Vert$ for all $i\in\Bbb I$, then  $\Vert x\Vert=\Vert y\Vert$.
\end{definition}
It is clear that if $\Lambda$ is a g-phase retrieval, then it is a g-norm retrieval. Also, if $\Lambda$ is a tight g-frame, then it dose g-norm retrieval. Indeed, if $A$ is a tight g-frame bound for $\Lambda$ and $\Vert\Lambda_i x\Vert=\Vert\Lambda_i y\Vert$ for each $x,y\in H$ and $i\in\Bbb I$, we have
\begin{align*}
A\Vert x\Vert^2=\sum_{i\in\Bbb I}\Vert \Lambda_{i}x\Vert^2=\sum_{i\in\Bbb I}\Vert \Lambda_{i}y\Vert^2=A\Vert y\Vert^2,
\end{align*}
and we get $\Vert x\Vert=\Vert y\Vert$.
\begin{example}
Let $H=\Bbb R^2$ and $\{e_1,e_2\}$ be a standard basis. For each $x\in H$, we define
\begin{align*}
\Lambda_1: & H \rightarrow \Bbb C , \qquad \Lambda_1 x=\langle x,e_1\rangle,\\
\Lambda_2: & H \rightarrow \Bbb C , \qquad \Lambda_2 x=\langle x,e_2\rangle,\\
\Lambda_3: & H \rightarrow \Bbb C , \qquad \Lambda_3 x=\langle x,e_1+e_2\rangle.
\end{align*}
So, it is obvious that $\lbrace\Lambda_1, \Lambda_2,\Lambda_3\rbrace$ is a g-frame for $H$ with bounds $1$ and $3$. Now, assume that $\Vert\Lambda_i x\Vert=\Vert\Lambda_i y\Vert$ where $x=(x_1,x_2)$, $y=(y_1,y_2)$ and $i=1,2,3$. Therefore, we get
\begin{align*}
\begin{cases}
\vert x_1\vert=\vert y_1\vert,\\
\vert x_2\vert=\vert y_2\vert,\\
\vert x_1+x_2\vert=\vert y_1+y_2\vert.
\end{cases}
\end{align*}
Then, the set $\lbrace\Lambda_1, \Lambda_2,\Lambda_3\rbrace$ is a g-norm retrieval and g-phase retrieval with $\lambda=\pm 1$.
\end{example}
\begin{example}
Let $H=\Bbb R^3$. For each $x=(x_1,x_2,x_3)\in H$, we define
\begin{align*}
\Lambda_1: & H \rightarrow  H , \qquad \Lambda_1 x=(x_1,0,0),\\
\Lambda_2: & H \rightarrow  H , \qquad \Lambda_2 x=(0,x_2,0),\\
\Lambda_3: & H \rightarrow  H , \qquad \Lambda_3 x=(0,0,x_3),\\
\Lambda_4: & H \rightarrow  H , \qquad \Lambda_4 x=(0,0,x_1+x_2+x_3).
\end{align*}
So, it is obvious that $\lbrace\Lambda_1, \Lambda_2,\Lambda_3,\Lambda_4\rbrace$ is a g-frame for $H$ with bounds $1$ and $4$. Suppose that $\Vert\Lambda_i x\Vert=\Vert\Lambda_i y\Vert$ such that $x=(x_1,x_2,x_3)$, $y=(y_1,y_2,y_3)$ and $i=1,2,3,4$.
Then, the set $\lbrace\Lambda_1, \Lambda_2,\Lambda_3,\Lambda_4\rbrace$ is a g-norm retrieval and g-phase retrieval with $\lambda=\pm 1$.
\end{example}
\begin{example}
Let $H=\Bbb R^3$ and $\{e_1,e_2,e_3\}$ be a standard basis. We define
\begin{align*}
\Lambda_1: & H \rightarrow  H , \qquad \Lambda_1 e_1=e_2, \qquad \Lambda_1 e_2=e_3,\qquad \Lambda_1 e_3=e_1,\\
\Lambda_2: & H \rightarrow  H , \qquad \Lambda_2 e_1=e_1, \qquad \Lambda_2 e_2=e_3,\qquad \Lambda_2 e_3=0.
\end{align*}
So, it is obvious that $\lbrace\Lambda_1, \Lambda_2\rbrace$ is a g-frame for $H$ with bounds $1$ and $2$. But, if $\Vert\Lambda_2 x\Vert=\Vert\Lambda_2 y\Vert$ where $x=(x_1,x_2,x_3)$, $y=(y_1,y_2,y_3)$, therefore $\Vert x\Vert\neq\Vert y\Vert$ and so the set dose not g-norm retrieval and g-phase retrieval.
\end{example}
\begin{example}
Let $H=\Bbb R^2$ and for any $x=(x_1,x_2)$ define
\begin{align*}
\Lambda_1: & H \rightarrow  H , \qquad \Lambda_1 (x_1,x_2)=(x_1,x_2),\\
\Lambda_2: & H \rightarrow  H , \qquad \Lambda_2 (x_1,x_2)=(x_2,-x_1).
\end{align*}
So, it is obvious that $\lbrace\Lambda_1, \Lambda_2\rbrace$ is a g-tight frame for $H$ with the bound  $2$ and so the set dose  g-norm retrieval, but the set dose not g-phase retrieval.
\end{example}
\begin{theorem}\label{th1}
Let $\lbrace \Lambda_i U\in\mathcal{B}(H,H_i)\rbrace_{i\in\Bbb I}$ be a g-norm retrieval for all $U\in\mathcal{GL}(H)$. Then $\lbrace \Lambda_i\in\mathcal{B}(H,H_i)\rbrace_{i\in\Bbb I}$ is also a g-norm retrieval.
\end{theorem}
\begin{proof}
Suppose that $x,y\in H$ and $\Vert \Lambda_i x\Vert=\Vert \Lambda_i y\Vert$ for any $i\in\Bbb I$. For any $U\in\mathcal{GL}(H)$ there  exist $x',y'\in H$ such that $Ux'=x$ and $Uy'=y$. Thus $\Vert \Lambda_i Ux'\Vert=\Vert \Lambda_i Uy'\Vert$. Therefore, via the assumption, we get  $\Vert  Ux'\Vert=\Vert  Uy'\Vert$ and so $\Vert x\Vert=\Vert y\Vert$.
\end{proof}
\begin{lemma}\label{lm1}
If $\Vert Ux\Vert=\Vert Uy\Vert$ for any $x,y\in H$ and each $U\in\mathcal{GL}(H)$, then there exists a scalar $\lambda$ such that $y=\lambda x$ and $\vert\lambda\vert=1$.
\end{lemma}
\begin{proof}
Let $k\in\Bbb I$ and $x,y\in H$ where $x\neq0$ and $\Vert Ux\Vert=\Vert Uy\Vert$ for each $U\in\mathcal{GL}(H)$. Choose an orthonormal basis $\{e_i\}_{i\in\Bbb I}$ for $H$ with $e_k=\dfrac{x}{\Vert x\Vert}$ and suppose that $y=\sum_{i\in I}y_i e_i$. Define $U:H\rightarrow H$ with $Ue_k=e_k$ and $Ue_i=\frac{1}{2}e_i$ for any $i\in\Bbb I$ and $i\neq k$. So, $U\in\mathcal{GL}(H)$ and we get
$$Uy=y_ke_k+\frac{1}{2}\sum_{i\neq k}y_ie_i.$$
Now, we have
\begin{align*}
\Vert x\Vert^2=\Vert Ux\Vert^2=\Vert Uy\Vert^2=y_k^2+\frac{1}{4}\sum_{i\neq k}y_i^2.
\end{align*}
On the other hand,
\begin{align*}
\Vert x\Vert^2=\Vert y\Vert^2=y_k^2+\sum_{i\neq k}y_i^2.
\end{align*}
Therefore, $\sum_{i\neq k}y_i^2=0$ and so $y_i=0$ for any $i\neq k$. Then, we get
$$y=y_k e_k=\frac{y_k}{\Vert x\Vert}x,$$
such that $\dfrac{\vert y_k\vert}{\Vert x\Vert}=1$.
\end{proof}
\begin{theorem}
The following conditions are equivalent.
\begin{enumerate}
\item[(I)] $\lbrace \Lambda_i\in\mathcal{B}(H,H_i)\rbrace_{i\in\Bbb I}$ is a g-phase retrieval.\\
\item[(II)] $\lbrace \Lambda_i U\in\mathcal{B}(H,H_i)\rbrace_{i\in\Bbb I}$ is a g-phase retrieval for all  $U\in\mathcal{GL}(H)$.\\
\item[(III)] $\lbrace \Lambda_i U\in\mathcal{B}(H,H_i)\rbrace_{i\in\Bbb I}$ is a g-norm retrieval for all $U\in\mathcal{GL}(H)$.\\
\end{enumerate}
\end{theorem}
\begin{proof}
$(I)\Rightarrow(II)$. Let $i\in\Bbb I$ and $x,y\in H$. If $\Vert \Lambda_i Ux\Vert=\Vert \Lambda_i Uy\Vert$ where $U\in\mathcal{B}(H)$ is a invertible, so there exists a scalar $\lambda$ such that $Ux=\lambda Uy$ and $\vert\lambda\vert=1$. Then $x=\lambda y$.

$(II)\Rightarrow(III)$ is clear.

$(III)\Rightarrow(I)$. Assume that  $x,y\in H$ and $\Vert \Lambda_i x\Vert=\Vert \Lambda_i y\Vert$ for any $i\in\Bbb I$. Then  $\Vert \Lambda_i UU^{-1}x\Vert=\Vert \Lambda_i UU^{-1}y\Vert$ for each $U\in\mathcal{B}(H)$. So, via the assumption, we get $\Vert U^{-1}x\Vert=\Vert U^{-1}y\Vert$ for each $U\in\mathcal{B}(H)$ and by Lemma \ref{lm1} the proof is completed.
\end{proof}
\begin{theorem}\label{th2}
Let $H=\Bbb R^n$ and
 $\Lambda$ is a g-norm retrieval on $\Bbb R^n$. For any partition $\{I_i\}_{i=1,2}$ of $\Bbb I$ we have $\spa\{\Lambda^{\ast}_i(H_i)\}_{i\in I_1}^{\perp}\perp \spa\{\Lambda^{\ast}_i(H_i)\}_{i\in I_2}^{\perp}$.
\end{theorem}
\begin{proof}
Assume that $\{I_i\}_{i=1,2}$ is a partition of $\Bbb I$, also $x\in \spa\{\Lambda^{\ast}_i(H_i)\}_{i\in I_1}^{\perp}$ and $y\in \spa\{\Lambda^{\ast}_i(H_i)\}_{i\in I_2}^{\perp}$. For each $i\in I$, it is easy to check that
$$\Vert\Lambda_i( x+y)\Vert^2=\Vert \Lambda_i x\Vert^2+\Vert \Lambda_i y\Vert^2=\Vert\Lambda_i( x-y)\Vert^2.$$
So, $\Vert\Lambda_i( x+y)\Vert=\Vert\Lambda_i( x-y)\Vert$, and by assumption, $\Vert x+y\Vert=\Vert x-y\Vert$. This implies that $\langle x,y\rangle=0$.
\end{proof}
In general, if $\Lambda$ is a g-tight frame on $H$ and a g-norm retrieval, then Theorem \ref{th2} is holds. Indeed, assume that $A$ is a bound of $\Lambda$, also $x\in \spa\{\Lambda^{\ast}_i(H_i)\}_{i\in I_1}^{\perp}$ and $y\in \spa\{\Lambda^{\ast}_i(H_i)\}_{i\in I_2}^{\perp}$. We can write, by \eqref{f1},
\begin{align*}
\langle x,y\rangle&=\sum_{i\in\Bbb I}\frac{1}{A^2}\langle \Lambda^{\ast}_i\Lambda_i x, y\rangle\\
&=\frac{1}{A^2}\Big\lbrace\sum_{i\in I_1}\langle \Lambda^{\ast}_i\Lambda_i x,  y\rangle+\sum_{i\in I_2}\langle \Lambda^{\ast}_i\Lambda_i x, y\rangle\Big\rbrace\\
&=\frac{1}{A^2}\Big\lbrace\sum_{i\in I_1}\langle  x, \Lambda^{\ast}_i\Lambda_i y\rangle+\sum_{i\in I_2}\langle \Lambda^{\ast}_i\Lambda_i x, y\rangle\Big\rbrace\\
&=0.
\end{align*}
Also, when  $\Lambda$ is a g-frame and $\Lambda_i S_{\Lambda}^{-1}=S_{\Lambda}^{-1}\Lambda_i $ for each $i\in\Bbb I$, then Theorem \ref{th2} is holds for $\Lambda$ on $H$. Indeed, if $x\in \spa\{\Lambda^{\ast}_i(H_i)\}_{i\in I_1}^{\perp}$ and $y\in \spa\{\Lambda^{\ast}_i(H_i)\}_{i\in I_2}^{\perp}$ again, then by \eqref{f2} we have,
\begin{align*}
\langle x,y\rangle&=\sum_{i\in\Bbb I}\langle \Lambda^{\ast}_i\Lambda_i S^{-1}x, y\rangle\\
&=\sum_{i\in I_1}\langle x, \Lambda^{\ast}_i\Lambda_i S^{-1} y\rangle+\sum_{i\in I_2}\langle \Lambda^{\ast}_i\Lambda_i S^{-1} x, y\rangle\\
&=0.
\end{align*}
\begin{theorem}\label{th4}
Let $H=\Bbb R^n$ and
 $\Lambda$ is a g-norm retrieval on $\Bbb R^n$. For  each non orthogonal vectors $x,y\in \Bbb R^n$ there exist $i\in\Bbb I$ and $h_i\in H_i$ such that  $\langle\Lambda^{\ast}_i(h_i),x\rangle\langle\Lambda^{\ast}_i(h_i), y\rangle\neq0$.
\end{theorem}
\begin{proof}
Suppose that there exist non orthogonal vectors $x,y\in\Bbb R^n$ such that
$\langle\Lambda^{\ast}_i(h_i),x\rangle\langle\Lambda^{\ast}_i(h_i), y\rangle=0$
 for each $h_i\in H_i$ and $i\in I$. Define, $I_1=\{i\in\Bbb I, \quad \langle\Lambda^{\ast}_i(h_i),x\rangle=0\}$ and $I_2=\{i\in\Bbb I, \quad \langle\Lambda^{\ast}_i(h_i),x\rangle=0\}$. Then $x\in \spa\{\Lambda^{\ast}_i(H_i)\}_{i\in I_1}^{\perp}$ and $y\in \spa\{\Lambda^{\ast}_i(H_i)\}_{i\in I_2}^{\perp}$, therefore by Theorem \ref{th2}, we get $\langle x,y\rangle=0$ and this is a contraction.
\end{proof}
The inverse of Theorem \ref{th2} can be hold with addition two condition on the operators $\Lambda_i$.
\begin{theorem}\label{th5}
Let $\Lambda=\{\Lambda_i\in\mathcal{B}(\Bbb R^n,\Bbb C)\}_{i\in I}$ and $I_1, I_2$ be arbitrary partition of $\Bbb I$.
If $\spa\{\Lambda_i^*(\Bbb C)\}^{\perp}_{i\in I_1}\perp\spa\{\Lambda_i^*(\Bbb C)\}^{\perp}_{i\in I_2}$ and all of the operators $\Lambda_i^*$ are surjective, then  $\Lambda$ is a g-norm retrieval on $\Bbb R^n$.
\end{theorem}
\begin{proof}
Suppose that $x,y\in\Bbb R^n$ and $i\in I$ so that $\Vert\Lambda_i x\Vert=\Vert\Lambda_i y\Vert$. Via Riesz representation theorem, we have a unique $z_i\in\Bbb R^n$ where $\Lambda_i x=\langle x,z_i\rangle$ and $\Lambda_i y=\langle y,z_i\rangle$.  Therefore, $\vert\langle x,z_i\rangle\vert=\vert\langle y,z_i\rangle\vert$ or $\langle x,z_i\rangle=\pm\langle y,z_i\rangle$. Now, we define $I_1=\{i\in\Bbb I, \quad \langle x,z_i\rangle=-\langle y,z_i\rangle\}$ and $I_2=\Bbb I/I_1$. So, $I_1$ and $I_2$ are a partition for $\Bbb I$ and also if $i\in I_1$ then $\langle x+y,z_i\rangle=0$ and for $i\in I_2$ we have $\langle x-y,z_i\rangle=0$. By hypothesis, there exist $w_i\in\Bbb C$ such that $z_i=\Lambda_i^* w_i$. Thus $x+y\in\spa\{\Lambda_i^* w_i\}^{\perp}_{i\in I_1}$ and $x-y\in\spa\{\Lambda_i^* w_i\}^{\perp}_{i\in I_2}$. This implies that $\langle x+y,x-y\rangle=0$ and we conclude that $\Vert x\Vert=\Vert y\Vert$.
\end{proof}
\begin{remark}
Assume that $U\in\mathcal{B}(H,K)$ and $U^*$ is a surjective operator. Since $\ker U=\mathcal{R}(U^*)^{\perp}$, then we can get $\ker U=\{0\}$ and so, $U$ is an injective operator. Thus, in Theorem \ref{th5}, the condition :"all of $\Lambda_i^*$ are surjective" can be changed by "all of $\Lambda_i$ are injective". But, in the Theorem, the surjective condition was important.
\end{remark}
\begin{theorem}\label{th6}
Let $\Lambda=\{\Lambda_i\in\mathcal{B}(\Bbb R^n,\Bbb C)\}_{i\in\Bbb I}$. Then the followings are equivalent:
\begin{enumerate}
\item $\Lambda$ is a g-norm retrievable.
\item For  each non orthogonal vectors $x,y\in \Bbb R^n$ there exist $i\in\Bbb I$ and $h\in\Bbb C$ such that
$$\langle\Lambda^{\ast}_i(h),x\rangle\langle\Lambda^{\ast}_i(h), y\rangle\neq0.$$
\item For any $\varepsilon\neq 0$, non-zero  $h\in\Bbb C$ and all of the operators $\Lambda_i^*$ are surjective, there exists a real number $C_0>0$ such that for all $x,y\in\Bbb R^n$, we have
\begin{align}\label{rrr}
R_{h,\varepsilon}(x,y):=\big\vert\langle x,y\rangle-\varepsilon\Vert x\Vert \Vert y\Vert\big\vert^2+\frac{1}{\Vert h\Vert^4}\sum_{i\in\Bbb I}\vert\langle\Lambda^{\ast}_i(h),x\rangle\vert^2 \vert\langle\Lambda^{\ast}_i(h), y\rangle\vert^2\geq C_0\Vert x\Vert^2 \Vert y\Vert^2.
\end{align}
\end{enumerate}
\end{theorem}
\begin{proof}
$(1)\Rightarrow (2)$ is clear by Theorem \ref{th4}.

$(2)\Rightarrow (3)$. Let $\varepsilon\neq 0$ and $0\neq h\in\Bbb C$. If $x=0$ or $y=0$, then \eqref{rrr} holds true. Define for each $x,y\in\Bbb R^n$,
$$F_{h,\varepsilon}(x,y)=\big\vert\langle x,y\rangle-\varepsilon\big\vert^2+\frac{1}{\Vert h\Vert^4}\sum_{i\in\Bbb I}\vert\langle\Lambda^{\ast}_i(h),x\rangle\vert^2 \vert\langle\Lambda^{\ast}_i(h), y\rangle\vert^2.$$
So, $F_{h,\varepsilon}$ is a continuous function on $\Bbb R^n\times\Bbb R^n$. Let
$$C_0:=\min_{(x,y)\in \mathcal{S}\times\mathcal{S}}F_{h,\varepsilon}(x,y),$$
where $\mathcal{S}=\{x\in\Bbb R^n,\quad \Vert x\Vert=1\}$. By (2) we get $C_0>0$ and for all non-zero vectors $x,y\in\Bbb R^n$, we have
$$R_{h,\varepsilon}(x,y)=\Vert x\Vert^2\Vert y\Vert^2 F_{h,\varepsilon}\Big(\frac{x}{\Vert x\Vert},\frac{y}{\Vert y\Vert}\Big)\geq C_0 \Vert x\Vert^2\Vert y\Vert^2.$$
$(3)\Rightarrow(1)$. Suppose that $\Lambda$ is not a g-norm retrievable for $\Bbb R^n$. Then, by Theorem \ref{th5}, there exist a partition $I_1,I_2$ for $\Bbb I$ and non-zero vectors $x\in \spa\{\Lambda^{\ast}_i(H_i)\}_{i\in I_1}^{\perp}$ and $y\in \spa\{\Lambda^{\ast}_i(H_i)\}_{i\in I_2}^{\perp}$ such that $\langle x,y\rangle\neq 0$. Now, for $\varepsilon:=\frac{\langle x,y\rangle}{\Vert x\Vert \Vert y\Vert}$ and $h=(1,0)$ we get
$R_{h,\varepsilon}(x,y)=0$. Thus $C_0=0$ and this is a contradiction.
\end{proof}
In the next result, the stability of g-norm retrievable for $\Bbb R^n$ will be discussed.
\begin{theorem}
Let $\Lambda=\{\Lambda_i\in\mathcal{B}(\Bbb R^n,\Bbb C)\}_{i\in\Bbb I}$ be a g-norm retrievable and g-Bessel sequence for $\Bbb R^n$ where $\vert\Bbb I\vert<\infty$, also all of the operators $\Lambda_i^*$ are surjective. Then there exists a $\lambda>0$ such that any set $\Theta:=\lbrace \Theta_i\in\mathcal{B}(\Bbb R^n,\Bbb C)\rbrace_{i\in\Bbb I}$ satisfying
$$\max_{i\in\Bbb I}\Vert\Lambda_i-\Theta_i\Vert<\lambda$$ is also a g-norm retrievable and g-Bessel sequence for $\Bbb R^n$.
\end{theorem}
\begin{proof}
Let $\varepsilon\neq 0$ and $0\neq h\in\Bbb C$. Via Theorem \ref{th6}, there exists $C_0>0$ so that,
\begin{align*}
R_{h,\varepsilon}(x,y):=\big\vert\langle x,y\rangle-\varepsilon\Vert x\Vert \Vert y\Vert\big\vert^2+\frac{1}{\Vert h\Vert^4}\sum_{i\in\Bbb I}\vert\langle\Lambda^{\ast}_i(h),x\rangle\vert^2 \vert\langle\Lambda^{\ast}_i(h), y\rangle\vert^2\geq C_0\Vert x\Vert^2 \Vert y\Vert^2,
\end{align*}
for each $x,y\in\Bbb R^n$. Assume that $B$ is the bound of g-Bessel for $\Lambda$ and first, we put
$0<\lambda<\vert\Bbb I\vert$. So, for any $x\in\Bbb R^n$ we can write
\begin{align*}
\sum_{i\in\Bbb I}\Vert\Lambda_i x-\Theta_i x\Vert^2<\vert\Bbb I\vert^2 \Vert x\Vert^2.
\end{align*}
Thus, by Lemma \ref{lemn}, the set $\Theta$ is a g-Bessel sequence for $\Bbb R^n$ with the bound $B':=\vert\Bbb I\vert+\sqrt{B}$. Now, put
\begin{align}
0<\lambda<\min\Big\{\vert\Bbb I\vert, \frac{C_0}{(B+B')(\sqrt{B}+\sqrt{B'})}\Big\},
\end{align}
and
\begin{align*}
R'_{h,\varepsilon}(x,y):=\big\vert\langle x,y\rangle-\varepsilon\Vert x\Vert \Vert y\Vert\big\vert^2+\frac{1}{\Vert h\Vert^4}\sum_{i\in\Bbb I}\vert\langle\Theta^{\ast}_i(h),x\rangle\vert^2 \vert\langle\Theta^{\ast}_i(h), y\rangle\vert^2.
\end{align*}
Notice that, $\Vert\Lambda^*_i h\Vert\leq\sqrt{B}\Vert h\Vert$ and $\Vert\Theta^*_i h\Vert\leq\sqrt{B}\Vert h\Vert$. Now, we have
\begin{align*}
R_{h,\varepsilon}(x,y)-R'_{h,\varepsilon}(x,y)&=\frac{1}{\Vert h\Vert^4}\sum_{i\in\Bbb I}\Big(\vert\langle\Lambda^{\ast}_i(h),x\rangle\vert^2 \vert\langle\Lambda^{\ast}_i(h), y\rangle\vert^2-\vert\langle\Theta^{\ast}_i(h),x\rangle\vert^2 \vert\langle\Theta^{\ast}_i(h), y\rangle\vert^2\Big)\\
&\leq\frac{1}{\Vert h\Vert^4}\sum_{i\in\Bbb I}\Big\vert\vert\langle\Lambda^{\ast}_i(h),x\rangle\vert \vert\langle\Lambda^{\ast}_i(h), y\rangle\vert-\vert\langle\Theta^{\ast}_i(h),x\rangle\vert \vert\langle\Theta^{\ast}_i(h), y\rangle\vert\Big\vert\\
&\qquad\times\Big\vert \vert\langle\Lambda^{\ast}_i(h),x\rangle\vert \vert\langle\Lambda^{\ast}_i(h), y\rangle\vert+\vert\langle\Theta^{\ast}_i(h),x\rangle\vert \vert\langle\Theta^{\ast}_i(h), y\rangle\vert\Big\vert.
\end{align*}
$Since,$
\begin{align*}
\Big\vert\vert\langle\Lambda^{\ast}_i(h),x\rangle\vert \vert\langle\Lambda^{\ast}_i(h), y\rangle\vert&-\vert\langle\Theta^{\ast}_i(h),x\rangle\vert \vert\langle\Theta^{\ast}_i(h), y\rangle\vert\Big\vert\\
&\leq\Big\vert\langle\Lambda^{\ast}_i(h),x\rangle \langle\Lambda^{\ast}_i(h), y\rangle-\langle\Theta^{\ast}_i(h),x\rangle\langle\Theta^{\ast}_i(h), y\rangle\Big\vert\\
&=\Big\vert\langle\Lambda^{\ast}_i(h),x\rangle \langle\Lambda^{\ast}_i(h)-\Theta_i^*(h), y\rangle-\langle\Lambda^{\ast}_i(h)-\Theta_i^*(h),x\rangle\langle\Theta^{\ast}_i(h), y\rangle\Big\vert\\
&\leq\Vert x\Vert \Vert y\Vert \Vert\Lambda^{\ast}_i(h)-\Theta_i^*(h)\Vert\Big(\Vert\Lambda^{\ast}_i(h)\Vert+\Vert\Theta_i^*(h)\Vert\Big)\\
&\leq \lambda(\sqrt{B}+\sqrt{B'})\Vert x\Vert \Vert y\Vert \Vert h\Vert^4,
\end{align*}
therefore, we get
\begin{align*}
R_{h,\varepsilon}&(x,y)-R'_{h,\varepsilon}(x,y)\\
&\leq\lambda(\sqrt{B}+\sqrt{B'})\Vert x\Vert \Vert y\Vert  \frac{1}{\Vert h\Vert^2}\sum_{i\in\Bbb I}\Big\vert \vert\langle\Lambda^{\ast}_i(h),x\rangle\vert \vert\langle\Lambda^{\ast}_i(h), y\rangle\vert+\vert\langle\Theta^{\ast}_i(h),x\rangle\vert \vert\langle\Theta^{\ast}_i(h), y\rangle\vert\Big\vert\\
&\leq\lambda(\sqrt{B}+\sqrt{B'})\Vert x\Vert \Vert y\Vert  \frac{1}{\Vert h\Vert^2}
\Bigg\lbrace\Big(\sum_{i\in\Bbb I}\vert\langle\Lambda^{\ast}_i(h),x\rangle\vert^2\Big)^{1/2}\Big(\sum_{i\in\Bbb I}\vert\langle\Lambda^{\ast}_i(h),y\rangle\vert^2\Big)^{1/2}\\
&\qquad +\Big(\sum_{i\in\Bbb I}\vert\langle\Theta^{\ast}_i(h),x\rangle\vert^2\Big)^{1/2}\Big(\sum_{i\in\Bbb I}\vert\langle\Theta^{\ast}_i(h),y\rangle\vert^2\Big)^{1/2}\Bigg\rbrace
\end{align*}
\begin{align*}
&=\lambda(\sqrt{B}+\sqrt{B'})\Vert x\Vert \Vert y\Vert  \frac{1}{\Vert h\Vert^2}
\Bigg\lbrace\Big(\sum_{i\in\Bbb I}\vert\langle h,\Lambda_i x\rangle\vert^2\Big)^{1/2}\Big(\sum_{i\in\Bbb I}\vert\langle h,\Lambda_i y\rangle\vert^2\Big)^{1/2}\\
&\qquad +\Big(\sum_{i\in\Bbb I}\vert\langle h,\Theta_i x\rangle\vert^2\Big)^{1/2}\Big(\sum_{i\in\Bbb I}\vert\langle h,\Theta_i y\rangle\vert^2\Big)^{1/2}\Bigg\rbrace\\
&\leq\lambda(\sqrt{B}+\sqrt{B'})\Vert x\Vert \Vert y\Vert
\Bigg\lbrace\Big(\sum_{i\in\Bbb I}\Vert\Lambda_i x\Vert^2\Big)^{1/2}\Big(\sum_{i\in\Bbb I}\Vert\Lambda_i y\Vert^2\Big)^{1/2}+\Big(\sum_{i\in\Bbb I}\Vert\Theta_i x\Vert^2\Big)^{1/2}\Big(\sum_{i\in\Bbb I}\Vert\Theta_i y\Vert^2\Big)^{1/2}\Bigg\rbrace\\
&\leq\lambda(\sqrt{B}+\sqrt{B'})(B+B')\Vert x\Vert^2 \Vert y\Vert^2.
\end{align*}
Thus, we conclude that
\begin{align*}
R'_{h,\varepsilon}(x,y)&\geq R_{h,\varepsilon}(x,y)-\lambda(\sqrt{B}+\sqrt{B'})(B+B')\Vert x\Vert^2 \Vert y\Vert^2\\
&\geq\Big(C_0-\lambda(\sqrt{B}+\sqrt{B'})(B+B')\Big)\Vert x\Vert^2 \Vert y\Vert^2,
\end{align*}
and by $(3)\Rightarrow(1)$ Theorem \ref{th6}, the proof is completed.
\end{proof}
%%%%%%%%%%%%%%%%%%%%%%%%%%%

\end{document}